\documentclass[11pt]{amsart}

\usepackage[english]{babel}
\usepackage{amsmath}
\usepackage{amssymb}
\usepackage{amsthm}
\usepackage{latexsym}
\usepackage{mathtools}
\usepackage{thmtools}
\usepackage{amsfonts}
\usepackage{mathrsfs}
\usepackage{textcomp}
\usepackage[T1]{fontenc}
\usepackage{graphicx}
\usepackage{setspace}
\usepackage{nicefrac}
\usepackage{indentfirst}
\usepackage{enumerate}
\usepackage{wasysym}
\usepackage{upgreek}
\usepackage[pdfpagelabels,hyperindex=false]{hyperref}
\hypersetup{
    colorlinks,
    linkcolor={red!50!black},
    citecolor={green!50!black},
    urlcolor={blue!80!black}
}
\usepackage{etoolbox}
\usepackage{mathdots}
\usepackage{wrapfig}
\usepackage{floatflt}
\usepackage{tensor} 
\usepackage{parskip}
\usepackage{lscape}
\usepackage{stmaryrd}
\usepackage[enableskew]{youngtab}
\usepackage{ytableau}
\usepackage{pifont}

\usepackage{tikz}
\usetikzlibrary{arrows,matrix}
\usetikzlibrary{positioning}
\usetikzlibrary{decorations}

\usepackage[all,cmtip]{xy}
\usepackage[labelfont=bf, font={small}]{caption}[2005/07/16]

\usepackage{xcolor}
\definecolor{red}{rgb}{1,0,0}

\newcommand{\vvirg}{ , \dots , }
\newcommand{\ootimes}{ \otimes \cdots \otimes }
\newcommand{\bboxtimes}{ \boxtimes \cdots \boxtimes }
\newcommand{\ooplus}{ \oplus \cdots \oplus }
\newcommand{\ttimes}{ \times \cdots \times }

\newcommand{\textsum}{{\textstyle \sum}}

\newcommand{\bfd}{\mathbf{d}}
\newcommand{\bfe}{\mathbf{e}}

\newcommand{\calO}{\mathcal{O}}

\newcommand{\calV}{\mathcal{V}}

\newcommand{\bbC}{\mathbb{C}}

\newcommand{\bbN}{\mathbb{N}}

\newcommand{\bbP}{\mathbb{P}}

\newcommand{\frakk}{\mathfrak{k}}

\newcommand{\rmR}{\mathrm{R}}

\newcommand{\bfdelta}{\bfmath{\delta}}

\renewcommand{\phi}{\varphi}
\newcommand{\eps}{\varepsilon}
\renewcommand{\theta}{\vartheta}

\renewcommand{\Im}{\mathrm{Im} \;}  
\DeclareMathOperator{\codim}{codim}

\DeclareMathOperator{\Sym}{Sym}

\newcommand{\bfmath}[1]{\mbox{\boldmath $#1$}}

\newcommand{\fillwidthof}[3][c]
	{%
		\parbox
		{%
			\widthof{#2}%
		}%
		{%
			\ifx#1c%
				\centering#3%
			\else\ifx#1l%
				#3\hfill%
			\else\ifx#1r%
				\hfill#3%
			\fi\fi\fi%
		}%
	}%

\def\mylettrine#1#2 {\lettrine{#1}{#2}\space}

\newcommand{\partinto}[1][]{\smash{\mathord{\mathchoice{%
  \xymatrix@=0.4em@1{%
  \ar@{|-}[rr]_-*--{\scriptstyle #1}
  &*{\phantom{\scriptstyle{#1}}}&}
}{
  \xymatrix@=0.25em@1{%
  \ar@{|-}[rr]_-*--{\scriptstyle #1}
  &*{\phantom{\scriptstyle{#1}}}&}
}{
  \xymatrix@=0.2em@1{%
  \ar@{|-}[rr]_-*--{\scriptscriptstyle #1}
  &*{\phantom{\scriptscriptstyle{#1}}}&}
}{}}}}
\newcommand{\partintonosmash}[1][]{\mathord{\mathchoice{%
  \xymatrix@=0.4em@1{%
  \ar@{|-}[rr]_-*--{\scriptstyle #1}
  &*{\phantom{\scriptstyle{#1}}}&}
}{
  \xymatrix@=0.25em@1{%
  \ar@{|-}[rr]_-*--{\scriptstyle #1}
  &*{\phantom{\scriptstyle{#1}}}&}
}{
  \xymatrix@=0.2em@1{%
  \ar@{|-}[rr]_-*--{\scriptscriptstyle #1}
  &*{\phantom{\scriptscriptstyle{#1}}}&}
}{}}}
\newcommand{\partintostar}[1][]{\smash{\mathord{\mathchoice{%
  \xymatrix@=0.4em@1{%
  \ar@{|-}[rr]_-*--{\scriptstyle #1}^-*--{\scriptstyle \ast}
  &*{\phantom{\scriptstyle{#1}}}&}
}{
  \xymatrix@=0.25em@1{%
  \ar@{|-}[rr]_-*--{\scriptstyle #1}^-*--{\scriptstyle \ast}
  &*{\phantom{\scriptstyle{#1}}}&}
}{
  \xymatrix@=0.2em@1{%
  \ar@{|-}[rr]_-*--{\scriptscriptstyle #1}^-*--{\scriptstyle \ast}
  &*{\phantom{\scriptscriptstyle{#1}}}&}
}{}}}}
\newcommand{\partintostarnosmash}[1][]{\mathord{\mathchoice{%
  \xymatrix@=0.4em@1{%
  \ar@{|-}[rr]_-*--{\scriptstyle #1}^-*--{\scriptstyle \ast}
  &*{\phantom{\scriptstyle{#1}}}&}
}{
  \xymatrix@=0.25em@1{%
  \ar@{|-}[rr]_-*--{\scriptstyle #1}^-*--{\scriptstyle \ast}
  &*{\phantom{\scriptstyle{#1}}}&}
}{
  \xymatrix@=0.2em@1{%
  \ar@{|-}[rr]_-*--{\scriptscriptstyle #1}^-*--{\scriptstyle \ast}
  &*{\phantom{\scriptscriptstyle{#1}}}&}
}{}}}

\usepackage[letterpaper,margin=1in]{geometry}

\newcommand{\rank}{\mathrm{R}}
\newcommand{\uR}{\underline{\mathrm{R}}}
\newcommand{\cR}{\mathrm{cR}}

\newcommand{\ucR}{\underline{\mathrm{cR}}}
\newcommand{\acR}{\mathrm{cR}^{{}^\otimes}}
\newcommand{\aR}{\mathrm{R}^{{}^\otimes}}
\newcommand{\cat}{\mathrm{cat}}

\newcommand{\rk}{\mathrm{rk}}

 \title{Border Waring Rank via Asymptotic Rank}

 \newtheorem{theorem}{Theorem}[section]
 \newtheorem{proposition}[theorem]{Proposition}
 \newtheorem{lemma}[theorem]{Lemma}
  \newtheorem{corollary}[theorem]{Corollary}

\theoremstyle{remark} 
  \newtheorem{remark}[theorem]{Remark}
 \newtheorem{example}[theorem]{Example}

\author[M. Christandl]{Matthias Christandl}
\author[F. Gesmundo]{Fulvio Gesmundo}
\address[M. Christandl, F. Gesmundo]{QMATH, University of Copenhagen, Universitetsparken 5, 2100 Copenhagen O., Denmark}

\author[A. Oneto]{Alessandro Oneto}
\address[A. Oneto]{Otto-von-Guericke Universit\"at Magdeburg, Germany}

\email[M. Christandl]{christandl@math.ku.dk}
\email[F. Gesmundo]{fulges@math.ku.dk}
\email[A. Oneto]{alessandro.oneto@ovgu.de, aless.oneto@gmail.com}


\begin{document}
 \begin{abstract}
We investigate an extension of a lower bound on the Waring (cactus) rank of homogeneous forms due to Ranestad and Schreyer. We show that for particular classes of homogeneous forms, for which a generalization of this method applies, the lower bound extends to the level of border (cactus) rank. The approach is based on recent results on tensor asymptotic rank.
 \end{abstract}
  
  \maketitle

  \thispagestyle{empty}
\section{Introduction}\label{sec: intro}
Let $f$ be a homogeneous polynomial, or form, of degree $d$ in $n+1$ variables over the complex numbers. The classical \textit{Waring problem} for homogeneous polynomials asks:
\begin{center}
	\textit{What is the minimum number of linear forms $\ell_i$ needed to write $f = \sum_{i=1}^r \ell_i^d$?}
\end{center}
This number is called the \textit{Waring rank}, or simply \textit{rank}, of $f$, denoted by $\rmR_d(f)$. From a geometric perspective, the Waring rank of $f$ is the \textit{rank with respect to the $d$-th Veronese embedding $\calV_{n,d}$ of $\bbP \bbC^{n+1}$ in $\bbP \Sym^d \bbC^{n+1}$}, where $\bbC^{n+1} \simeq \Sym^1 \bbC^{n+1}$ is the vector space of linear forms in $n+1$ variables $x_0\vvirg x_n$ and $\Sym^d \bbC^{n+1}$ is the vector space of homogeneous polynomials of degree $d$ in $x_0 \vvirg x_n$. Indeed, let $\nu_d : \bbP \bbC^{n+1} \to \bbP \Sym^d \bbC^{n+1}$ be the Veronese embedding defined by $\nu_d([\ell]) = [\ell^d]$, so that $\calV_{n,d} = \nu_d(\bbP \bbC^{n+1})$. Then, the Waring rank of $f \in \Sym^d \bbC^{n+1}$ is 
\begin{equation}\label{eqn: def of rank via points}
\rmR_d(f) = \min\left\{ r : \exists \ell_1,\ldots,\ell_r \in \bbC^{n+1} \text{ such that } [f] \in \langle \nu_d([\ell_1]) \vvirg \nu_d([\ell_r]) \rangle \subseteq \bbP \Sym^d \bbC^{n+1}\right\}. 
\end{equation}
The notion of Waring rank generalizes the rank of a matrix, and indeed, if $f$ is a quadratic form, its Waring rank coincides with the usual rank of the associated symmetric matrix. 

The Waring problem has a long history, dating back at least to Sylvester \cite{Syl51} who gave a complete solution in the case of binary forms, i.e., homogeneous polynomials in two variables. A complete answer has been given for generic forms for every degree and number of variables in the celebrated Alexander-Hirschowitz Theorem \cite{AH95:SolutionGenericWaring}. As far as specific forms are concerned, only few cases are known: relevant for this work is the answer in the case of monomials, provided by Carlini, Catalisano and Geramita \cite{CCG12:SolutionWaringMonomials}. Their result shows that if $m = x_0^{a_0}\cdots x_n^{a_n}$ is a monomial with $a_0 = \min_i\{a_i\}$, then
\[
 \rank_{d}(m) = \frac{1}{a_0+1}\prod_{i=0}^n (a_i+1).
\]
Over other fields the problem remains open. For monomials in two variables, it is known that the real Waring rank always equals the degree \cite{BCG:BinaryMonomials}. For monomials in more variables, it is known that the real Waring rank and the complex Waring rank coincide if and only if one of the exponents is equal to one \cite{CKOV17:Monomials}; to the extent of our knowledge, already the real Waring rank of the monomial $x_0^2x_1^2x_2^2$ is not known. Over other subfields of the complex numbers, we refer to \cite{RT:BinaryForms} for some partial result. We refer to \cite{BCCGO:HitchhikerGuide} for an extensive survey on the state-of-the-art on the subject from an algebraic and geometric point of view. 

The Waring rank fails to be upper semicontinuous, unlike matrix rank. More precisely, there are examples of sequences $(f_\eps)_{\eps > 0} \subset \Sym^d \bbC^{n+1}$ such that $\lim_{\eps \rightarrow 0}f_\eps = f$ but $\rmR_d(f_\eps) < \rmR_d(f)$, for any $\eps$. The very first example of this phenomenon, already known to Sylvester \cite{Sylv:PrinciplesCalculusForms}, is for the monomial $f = x_0x_1^2$: indeed $\rmR_{3}(f) = 3$ but the sequence $f_\eps = \frac{1}{3\eps}\left((\eps x_0 + x_1)^3 - x_1^3\right)$, for $\eps > 0$, converges to $f$ when $\eps \to 0$ and $\rmR_3(f_\eps) = 2$. The notion of \emph{border Waring rank}, introduced by Bini \cite{Bini:RelationsExactApproxBilAlg} in the setting of tensors, but essentially dating back to Terracini \cite{Ter15:RappresFormeQuaternarie}, is the semicontinuous closure of the Waring rank; for $f \in \Sym^d \bbC^{n+1}$, the \emph{border Waring rank} of $f$, or simply \emph{border rank}, denoted $\uR_{d}(f)$, is the minimum $r$ such that $f$ can be approximated by forms of rank $r$; more precisely
\[
\uR_{d}(f) = \min\left\{r : \exists (f_\eps)_{\eps > 0} \subset \Sym^d\bbC^{n+1} \text{ such that }f = \lim_{\eps \rightarrow 0} f_\eps \text{ with } \rank_{d}(f_\eps) = r\right\}.
\]
The notion of border Waring rank can be defined in terms of \textit{secant varieties of Veronese varieties}. The $r$-th secant variety of $\calV_{n,d}$, denoted $\sigma_r(\calV_{n,d})$, is the closure (equivalently in the Zariski or the Euclidean topology) of the union of all linear spaces spanned by $r$ points of $\calV_{n,d}$; in other words
\[
\sigma_r(\calV_{n,d}) = \left\{ [f] \in \bbP \Sym^d\bbC^{n+1} ~:~ \uR_{d}(f) \leq r\right\}. 
\]

The Waring rank and the border Waring rank have scheme-theoretic analog notions, called \emph{cactus (Waring) rank} and \emph{border cactus (Waring) rank}. The cactus rank and the border cactus rank of $f \in \Sym^d \bbC^{n+1}$ are defined, respectively, as follows:
\begin{align*}
 \cR_d(f) &= \min\left\{ r : \exists S \subset \bbP \bbC^{n+1} \text{ a $0$-dimensional scheme of length } r \text{ such that }[f] \in \langle \nu_d(S) \rangle\right\}; \\
 \ucR_d(f) &= \min\left\{ r : \exists (f_\eps)_{\eps > 0} \subset \Sym^d\bbC^{n+1} \text{ such that }f = \lim_{\eps \rightarrow 0} f_\eps \text{ with } \cR_{d}(f_\eps) = r\right\}.
\end{align*}
Note that the notion of cactus rank generalizes the notion of rank, where the $0$-dimensional $S$ is required to be also reduced, i.e., a set of points; see \eqref{eqn: def of rank via points}. In particular, we have the inequalities $\cR_d(f) \leq \rmR_d(f)$ and $\ucR_d(f) \leq \uR_d(f)$. Cactus rank and border rank are not comparable in general: there are examples of forms $f \in \Sym^d \bbC^{n+1}$ with $\uR_d(f) < \cR_d(f)$ (see, e.g., \cite[Example 2.8]{BerBraMou:ComparisonDifferentNotionsRanksSymmetricTensors}) and examples of forms with $\cR_d(f) < \uR_d(f)$ (see, e.g., \cite{BerRan:CactusRankCubicForm}).

One can define a scheme-theoretic analog of the secant variety, called \emph{cactus variety} \cite{BB14:Cactus}:
\[
 \frakk_r(\calV_{n,d}) = \left\{ [f] \in \bbP \Sym^d\bbC^{n+1} ~:~ \ucR_{d}(f) \leq r\right\}. 
\]

Little is known about the geometric structure of cactus varieties: the difficulty lies in the fact that the Hilbert scheme of points of $\bbP \bbC^{n+1}$ is not irreducible except for few small values of $n$, see \cite{CarErmVelVir:HilbertScheme8,BB14:Cactus}. In general, there is one component, the so-called \emph{smoothable component}, containing schemes $S$ arising as \emph{flat limit} of sets of distinct points, and several other components, whose points are so-called \emph{non-smoothable $0$-dimensional schemes}. These schemes are responsible of the strict inequalities $\ucR_d(f) < \uR_d(f)$, and $\cR_d(f) < \uR_d(f)$ when they occur. We refer to \cite{Jeli:PathologiesHilbertScheme} for an extensive explanation of this theory and of these phenomena.

Computing rank, border rank, cactus rank and border cactus rank of an explicit homogeneous polynomial is a hard problem. Upper bounds are often found by providing explicit expressions of the form. As far as lower bounds are concerned, several methods have been proposed in the literature. One of the classical techniques is based on \emph{flattening methods}. Informally, these consider linear maps from the space of forms to a space of matrices: from the lower bound on the matrix rank of the image of a form, one can deduce a lower bound on the Waring rank of the form; moreover, exploiting semicontinuity of matrix rank, one can show that these lower bounds pass to border rank. We refer to \cite{LanOtt:EqnsSecantVarsVeroneseandOthers} for the precise statements and to \cite[\S7.2]{BalBerChrGes:PartiallySymRkW} for a more general description of the method. In \cite{Galazka:VectorBundlesGiveEqnsForCactus} and related work, it was shown that flattening techniques give, in fact, lower bounds on the cactus rank, and by semicontinuity these lower bounds extend to border cactus rank.

A different method, based on apolarity theory, was introduced by Ranestad and Schreyer in \cite[Proposition~1]{RanSch:RankSymmetricForm}. We refer to Section \ref{subsec: apolarity} for the definitions regarding apolarity and to Section \ref{sec: RS} for an exposition of the Ranestad-Schreyer method. We mention here that the lower bound of \cite[Proposition~1]{RanSch:RankSymmetricForm} was stated for cactus rank, and hence for rank. In \cite[Theorem 5.13]{Tei14:GeometricBounds}, Teitler proposed a generalization of this bound to partially symmetric tensors (see section \ref{subsec: partially symmetric} for an introduction of partially symmetric ranks); however, the proof has a gap \cite{personalComm}. For this reason, we introduce the following condition for partially symmetric tensors, which identifies those tensors for which \cite[Theorem 5.13]{Tei14:GeometricBounds} holds. We say that a partially symmetric tensor $t \in S^{d_1} V_1 \ootimes S^{d_k} V_k$ for which the apolar ideal $t^\perp$ is generated in multidegree $\bfdelta = (\delta_1 \vvirg \delta_k)$ satisfies condition $(\star)$ if
\begin{equation}
\rmR_{\bfd}(t) \geq \frac{\dim A_t}{\delta_1 \cdots \delta_k} \tag{$\star$}
\end{equation}
where $A_t = \Sym( V_1 \ooplus V_k) / t^\perp$ is the multigraded apolar algebra of $t$. Analogously, we say that a tensor $t$ satisfies the condition $(\star_c)$ if 
\begin{equation}
\cR_{\bfd}(t) \geq \frac{\dim A_t}{\delta_1 \cdots \delta_k}. \tag{$\star_c$}
\end{equation}
For instance, symmetric tensors satisfy the condition $(\star_c)$ by \cite[Proposition~1]{RanSch:RankSymmetricForm}. Clearly, a tensor that satisfies condition $(\star_c)$ satisfies condition $(\star)$ as well.

\begin{theorem}\label{thm: RS for border rank}
 Let $f \in \Sym^d \bbC^{n+1}$ be a homogeneous polynomial such that $f^\perp$ is generated in degree $\delta$. If, for every $k$, $f^{\otimes k} \in (S^d \bbC^{n+1})^{\otimes k}$ satisfies condition $(\star)$ (resp. condition $(\star_c)$), then $\uR_{d}(f) \geq  \frac{\dim A_f}{\delta}$ (resp. $\ucR_d(f) \geq \frac{\dim A_f}{\delta})$).
\end{theorem}

\begin{remark}
 In a previous version of this paper, Theorem \ref{thm: RS for border rank} was stated unconditionally to conditions $(\star)$ and $(\star_c)$. The multiplicativity result was based on \cite[Theorem 5.13]{Tei14:GeometricBounds}. It was pointed out to us \cite{personalComm} that the proof of \cite[Theorem 5.13]{Tei14:GeometricBounds} has a gap, which in turn made the proof of the original Theorem \ref{thm: RS for border rank} incomplete. For this reason, we restated it introducing the conditions  $(\star)$ and $(\star_c)$. Anyway, although the proof of \cite[Theorem 5.13]{Tei14:GeometricBounds} is incomplete, we do not have any counterexample to its statement.
 \end{remark}

Our approach is based on the following two fundamental building blocks.
\begin{itemize}
 \item The quantity $\frac{\dim A_f}{\delta}$ is multiplicative under tensor product. In other words, for $j = 1 \vvirg k$, suppose $f_j$ is a homogeneous polynomials with apolar ideal is generate in degree $\delta_j$. Let $t = f_1 \ootimes f_k$: then $t^\perp$ is generated in multidegree $\bfdelta = (\delta_1 \vvirg \delta_k)$ and $\dim A_t = \dim A_{f_1} \cdots \dim A_{f_k}$. See 
  Proposition \ref{prop: RS multiplicative} and Corollary \ref{corol: multiplicative RS}.
 \item Conditions $(\star)$ and $(\star_c)$ guarantee that $\frac{\dim A_f }{\delta}$ is a multiplicative lower bound for the rank and the cactus rank of $f$, respectively. It turns out that multiplicative rank lower bounds for rank ``pass to the border rank''. More precisely, multiplicative rank lower bounds are lower bounds for the tensor asymptotic rank introduced in \cite{ChrGesJen:BorderRankNonMult}, see Lemma~\ref{lemma:MultiplicativeLowerBound}, and in turn lower bounds for border rank by \cite[Proposition~6.2]{ChrGesJen:BorderRankNonMult} and \cite[Theorem~8]{ChrJenZui:NonMultTensorRank}. As for cactus rank, we observe that an analog of \cite[Proposition~6.2]{ChrGesJen:BorderRankNonMult} holds for the cactus variety, see Proposition \ref{prop: lower bound to border cactus}: this extends multiplicative lower bounds for cactus rank to border cactus rank.
\end{itemize}

We believe that this approach can be used to solve the problem of border rank for monomials. It has been conjectured for almost a decade that if $m = x_0^{a_0}\cdots x_n^{a_n}$ is a monomial, in particular,
\[
\ucR_{d}(m) = \uR_{d}(m) = \frac{1}{a_n+1}\prod_{i=0}^n (a_i+1).
\]

The upper bound $\uR_d(m) \leq \frac{1}{a_n+1}\prod_{i=0}^n (a_i+1)$ was proved by Landsberg and Teitler in \cite[Theorem 11.2]{LaTe10:RanksBorderRanks} and equality was shown under the assumption $a_n \geq a_0+ \cdots + a_{n-1}$. In \cite{Oe16:BorderMonomials}, Oeding proved equality for a number of other families of monomials. These results achieved the desired lower bound via flattening methods. At the UMI-SIMAI-PTM Joint Meeting in Wroc{\l}aw~(PL) in September 2018, Buczy{\'n}ski announced a proof for a number of other cases obtained by using an original method based on apolarity theory \cite{BucBuc:BorderApolarity}, but the problem is still open in general.

 The equality $\cR_d(m) = \frac{1}{a_n+1}\prod_{i=0}^n (a_i+1)$ holds for cactus rank \cite{RanSch:RankSymmetricForm}. In particular, if the conditions $(\star_c)$ holds in the case of tensor powers of monomials, then cactus rank, cactus border rank and border rank of monomials coincide. 

\section{Preliminary results}\label{sec: preliminaries}
We first recall the partially symmetric versions of rank and border rank, and we discuss apolarity theory, both in the homogeneous and the multihomogeneous setting. We provide some results on multiplicative lower bounds, their relations with tensor asymptotic rank and to its cactus~analog.

\subsection{Partially symmetric tensors}\label{subsec: partially symmetric}

Let $V_1 \vvirg V_k$ be complex vector spaces of dimension $n_1+1 \vvirg n_k + 1$, respectively; for each $i = 1,\ldots,k$, write $\{ x_{i,j} : j = 0 \vvirg n_i\}$ for a basis of $V_i$. Consider the ring of polynomials in all the variables $x_{ij}$, that is $$\bbC[x_{i,j}: i = 1 \vvirg k, j= 0 \vvirg n_i] \simeq \Sym^\bullet (V_1 \ooplus V_k),$$ identified with the symmetric algebra of $V_1 \ooplus V_k$. Then, $$\Sym^\bullet (V_1 \ooplus V_k) \simeq  \Sym^\bullet V_1 \ootimes \Sym ^\bullet V_k,$$ where $\Sym^\bullet V_i$ is the ring of polynomials in the variables $x_{i,0} \vvirg x_{i, n_i}$.

In particular, $\bbC[x_{ij}: i = 1 \vvirg k, j= 0 \vvirg n_i]$ inherits the natural multigrading given by the tensor products of the symmetric algebras
\[
 \Sym^\bullet (V_1 \ooplus V_k) \simeq \bigoplus_{d_1 \vvirg d_k \geq 0} \Sym^{d_1} V_1 \ootimes \Sym^{d_k} V_k,
 \]
i.e., $\deg(x_{i,j}) = (0,\ldots,1,\ldots,0) \in \bbN^k$, where the $1$ is in the $i$-th entry. A partially symmetric tensor of multidegree $\bfd = (d_1 \vvirg d_k)$ is an element of the multigraded component $\Sym^{d_1} V_1 \ootimes \Sym^{d_k} V_k$ or, equivalently, a multihomogeneous polynomial of multidegree $\bfd$ in $\bbC[x_{i,j}]$.

If $t \in \Sym^{d_1} V_1 \ootimes \Sym^{d_k} V_k$ is a partially symmetric tensor, the \emph{partially symmetric rank} of $t$ is defined as
\[
 \rmR_{\bfd} (t) = \min \left\{ r ~:~ t = \textsum _1^r \ell_{1,j}^{d_1} \ootimes \ell_{k,j}^{d_k} , \text{ for some $\ell_{i,j} \in \Sym^1V_i$}\right\}.
\]
This is the rank with respect to the \emph{Segre-Veronese variety} $\calV_{d_1,n_1}\ttimes\calV_{d_k,n_k}$ obtained as the embedding of $\bbP V_1 \ttimes \bbP V_k$ via the linear system $|\calO_{\bbP V_1}(d_1) \ootimes \calO_{\bbP V_k}(d_k)|$ in the projective space $\bbP (\Sym^{d_1} V_1 \ootimes \Sym^{d_k} V_k)$, i.e., via the embedding 
\begin{align*}
	\nu_\bfd : \bbP V_1 \ttimes \bbP V_k &\longrightarrow \bbP (\Sym^{d_1} V_1 \ootimes \Sym^{d_k} V_k) \\
	([\ell_1],\ldots,[\ell_k]) &\longmapsto \left[\ell_1^{d_1}\ootimes\ell_k^{d_k}\right].
\end{align*}
In the case $k = 1$, this coincides with the Waring rank. In the case $k = 2,d_1=d_2=1$, the tensor $t$ is a bilinear form and the partially symmetric rank coincides with the rank of the associated matrix.

Similarly to the homogeneous setting described before, the \emph{partially symmetric border rank} of a tensor $t \in \Sym^{d_1} V_1 \ootimes \Sym^{d_k} V_k$ is the minimum $r$ such that $t$ can be approximated by partially symmetric tensors of rank $r$, namely
\[
 \uR_{\bfd}(t) = \min \left\{ r ~:~ \exists (t_\eps)_{\eps > 0} \subset \Sym^{d_1} V_1 \ootimes \Sym^{d_k} \text{ such that } t = \lim_{\eps \to 0} t_\eps \text{ with } \rmR_{\bfd}(t_\eps) = r \right\}.
\]
Equivalently, the number $\uR_{\bfd}(t)$ is the smallest $r$ such that $[t] \in\bbP (\Sym^{d_1} V_1 \ootimes \Sym^{d_k} V_k)$ belongs to the $r$-th secant variety $ \sigma_r( \calV_{d_1,n_1} \ttimes \calV_{d_k,n_k} )$ of the Segre-Veronese variety.

It is straightforward to verify that partially symmetric rank and border rank are submultiplicative under tensor product in the following sense. Let $t,s$ be partially symmetric tensors, say $t \in \Sym^{d_1} V_1 \ootimes \Sym^{d_k} V_k$ and $s \in \Sym^{e_1} V_{k+1} \ootimes \Sym^{e_\ell} V_{k+\ell}$. Then
\[
\rmR_{\bfd | \bfe}(t \otimes s) \leq \rmR_{\bfd}(t) \cdot \rmR_{\bfe}(s) \quad \text{and} \quad \uR_{\bfd | \bfe}(t \otimes s) \leq \uR_{\bfd}(t) \cdot \uR_{\bfe}(s), 
\]
where $\bfd | \bfe$ denotes the concatenation of $\bfd = (d_1,\ldots,d_k)$ and $\bfe = (e_1,\ldots,e_\ell)$.

Both inequalities can be strict, as observed in \cite{ChrJenZui:NonMultTensorRank,ChrGesJen:BorderRankNonMult,BalBerChrGes:PartiallySymRkW}.

Cactus analog of rank and border rank can be defined in the partially symmetric setting as well. Given $t \in \Sym^{d_1} V_1 \ootimes \Sym^{d_k}V_k$, define
\begin{align*}
 \cR_\bfd(t) &= \min\left\{ r : \exists S \subset \bbP V_1 \ttimes \bbP V_k \text{ a $0$-dim. scheme of length } r \text{ such that }[t] \in \langle \nu_\bfd(S) \rangle\right\}; \\
 \ucR_d(t) &= \min\left\{ r : \exists (t_\eps)_{\eps > 0} \subset \Sym^{d_1} V_1 \ootimes \Sym^{d_k} \text{ such that } t = \lim _{\eps \to 0} t_\eps \text{ with } \cR_\bfd(t_\eps) = r\right\}. 
\end{align*}
Submultiplicativity holds for cactus rank and border rank, as well. Let $t,s$ be partially symmetric tensors, say $t \in \Sym^{d_1} V_1 \ootimes \Sym^{d_k} V_k$ and $s \in \Sym^{e_1} V_{k+1} \ootimes \Sym^{e_\ell} V_{k+\ell}$. Then,
\[
\cR_{\bfd | \bfe}(t \otimes s) \leq \cR_{\bfd}(t) \cdot \cR_{\bfe}(s) \quad \text{and} \quad \ucR_{\bfd | \bfe}(t \otimes s) \leq \ucR_{\bfd}(t) \cdot \ucR_{\bfe}(s).
\]

\subsection{Apolarity theory}\label{subsec: apolarity}

Apolarity is a classical approach to the Waring problem: it dates back to Sylvester \cite{Sylv:PrinciplesCalculusForms} and it has been used, directly or indirectly, to achieve most of the known results for Waring rank of specific forms.

We briefly present the subject and we refer the reader to \cite{IarrKan:PowerSumsBook, Gera:InvSysFatPts, CarGriOed:FourLecturesSecantVarieties,BCCGO:HitchhikerGuide} for a complete explanation of this material in the homogeneous setting and to \cite{Tei14:GeometricBounds,Galazka:MultigradedApolarity,GalRanVil:VarsApolarSubschToricSurfaces} for the multigraded version.

Let $V$ be a vector space of dimension $n+1$ with a basis $\{x_0 \vvirg x_n\}$ and let $\Sym^\bullet V$ be the symmetric algebra of $V$, identified with the standard graded ring of polynomials $\bbC[x_0 \vvirg x_n]$. Let $V^*$ be the dual vector space with basis $\{\partial_0,\ldots,\partial_n\}$. The symmetric algebra $\Sym^\bullet V^*$ can be regarded as the ring of differential operators with constant coefficients, where $\partial_i$ corresponds to $\frac{\partial}{\partial x_i}$, for all $i = 0,\ldots,n$. Hence, we have a natural action of $\Sym^\bullet V^*$ on $\Sym^\bullet V$ via differentiation: 
\[
\begin{array}{c c c c}
	\circ : & \Sym^\bullet V^* \times \Sym^\bullet V & \longrightarrow & \Sym^\bullet V, \\
	& (D,f) & \longmapsto & D \circ f := D(f).
\end{array}
\]

For a homogeneous polynomial $f \in \Sym^\bullet V$, the \emph{apolar ideal} of $f$ is defined by 
\[
 f^\perp = \left\{D \in \Sym^\bullet V^* : D \circ f = 0\right\}.
\]
It is clear that $f^\perp$ is a homogeneous ideal in $ \Sym^\bullet V^*$.

Given $f \in \Sym^d V$, the \textit{$e$-th catalecticant} of $f$ is the linear map
\[
 \begin{array}{c c c c}
	\cat_e(f) : & \Sym^e V^* & \longrightarrow & \Sym^{d-e} V, \\
	& D & \longmapsto & D \circ f.
\end{array}
\]
By definition, the homogeneous component of degree $e$ in $f^\perp$ coincides with the kernel of the $e$-th catalecticant map, namely $(f^\perp)_{e} = \ker ( \cat_e(f))$, and, in particular, we have $(f^\perp)_{e} = \Sym^e V^*$, for $e > \deg(f)$.

\begin{example}
	Let $m = x_0^{a_0}\cdots x_n^{a_n}$ be a monomial; then $m^\perp = (\partial_0^{a_0+1},\ldots,\partial_n^{a_n+1})$.
\end{example}

The apolarity action and the notion of apolar ideal extend to the multigraded setting and to tensor products (and even in more general settings, see \cite{Galazka:MultigradedApolarity}). The polynomial ring $\Sym ^\bullet (V_1^* \ooplus V_k^*)$ acts again via apolarity on the polynomial ring $\Sym ^\bullet (V_1 \ooplus V_k)$. If $t \in \Sym^{d_1} V_1 \ootimes \Sym^{d_k} V_k$ is a partially symmetric tensor, then $t^\perp$ is a multihomogeneous ideal in $\Sym ^\bullet (V_1^* \ooplus V_k^*)$. Now, for every $\bfe = (e_1 \vvirg e_k)$, define a multigraded catalecticant map:
\[
\begin{array}{r c r c}
	\cat_{\bfe}(t) : & \Sym^{e_1} V_1^*\ootimes \Sym^{e_k} V_k^* & \to & \Sym^{d_1-e_1} V_1 \ootimes \Sym^{d_k-e_k} V_k, \\
	& D & \mapsto & D \circ t = D(t),
\end{array}
\]
where $D$ and $t$ are regarded as elements of $\Sym ^\bullet (V_1^* \ooplus V_k^*)$ and $\Sym ^\bullet (V_1 \ooplus V_k)$, respectively. Then, the multihomogeneous components of the apolar ideal $t^\perp$ coincide with the kernels of the catalecticant maps, i.e., $(t^\perp)_\bfe = \ker(\cat_\bfe(t))$. Similarly to the homogeneous case, we have that $(t^\perp)_\bfe = \Sym^{e_1} V_1^*\ootimes \Sym^{e_k} V_k^*$ if $e_j > d_j$, for at least one $j$.

In the special case where $t = f_1 \ootimes f_k \in \Sym^{d_1} V_1 \ootimes \Sym^{d_k} V_k$ there is a direct characterization of the apolar ideal of $t$ in terms of the apolar ideals of the $f_i$'s.

Let us first recall a basic lemma from linear algebra that we will use in the proof.

\begin{lemma}\label{lemma:Kronecker}
Let $A_1,\ldots,A_k$ be linear maps, with $A_i : V_i \to W_i$.  Let $A_1 \bboxtimes A_k$ be the Kronecker product of the $A_i$'s, i.e., the linear map defined by 
\begin{align*}
 A_1 \bboxtimes A_k : V_1 \ootimes V_k &\to W_1 \ootimes W_k \\
 v_1 \ootimes v_k &\mapsto A_1(v_1) \ootimes A_k(v_k),
\end{align*}
and extended linearly. Then:
	\begin{enumerate}[{\rm (i)}]
	\item $\rk(A_1	\bboxtimes A_k) = \rk( A_1) \cdots \rk(A_k)$;
	\item $\ker(A_1	\bboxtimes A_k) = \ker(A_1)\otimes V_2 \ootimes V_k + \cdots + V_1 \ootimes V_{k-1} \otimes \ker(A_k)$.
	\end{enumerate}
\end{lemma}
\begin{proof}
We prove the result in the case $k = 2$. The general result follows by induction.
\begin{enumerate}[(i)]
\item It is enough to show that $\Im(A_1 \boxtimes A_2) = \Im(A_1) \otimes \Im(A_2)$. This is immediate from the fact that tensor products are generated by product elements.
\item The inclusion of the right-hand side in the left-hand side is immediate. The other inclusion follows by a dimension argument. Let $r_i := \rk(A_i)$. From part (i), $\dim \ker(A_1 \boxtimes A_2) = m_1m_2 - \rk(A_1 \boxtimes A_2) = m_1m_2 - r_1r_2$. On the other hand, by Grassmann's formula,
		\begin{align*}
			 \dim (\ker(A_1)\otimes V_2 & + V_1\otimes\ker(A_2)) = \\
			& = (m_1-r_1)m_2 + m_1(m_2-r_2) - (m_1-r_1)(m_2-r_2) \\
			& = m_1m_2 - r_1r_2.
		\end{align*}
	\end{enumerate}
\end{proof}
Now, notice that for every $j = 1\vvirg k$, the polynomial ring $\Sym^\bullet V_j^*$ can be regarded as a subring of $\Sym^\bullet (V_1^* \ooplus V_k^*)$. If $I \subseteq \Sym^\bullet V_j^*$ is an ideal for some $j = 1 \vvirg k$, write $I^{ext}$ for the ideal generated by $I$ in $\Sym^\bullet (V_1^* \ooplus V_k^*)$.

  \begin{lemma}\label{lemma:apolar}
Let $t = f_1\otimes\ldots\otimes f_k \in \Sym^{d_1}V_1 \otimes \cdots \otimes \Sym^{d_k}V_k$. Then, 
\[
t^\perp = (f_1^\perp)^{ext} + \ldots + (f_k^\perp)^{ext}. 
\]
  \end{lemma}
  \begin{proof}
  We show that the two ideals coincide in every multidegree $\bfe = (e_1 \vvirg e_k)$.
  
   Recall that $(t^\perp)_{\bfe} = \ker (\cat_{\bfe} (t))$. Moreover, we have $$(f_i^\perp)^{ext}_{\bfe} = \Sym^{e_1} V_1^* \ootimes \Sym^{e_{i-1}} V_{i-1}^* \otimes (f_i^\perp)_{e_i} \otimes \Sym^{e_{i+1}} V_{i+1}^* \ootimes \Sym^{e_k} V_k^*$$ and recall that $(f_i^\perp)_{e_i} = \ker (\cat_{e_i}(f_i))$. 
  
  By Lemma \ref{lemma:Kronecker}(ii), it suffices to show that $\cat_{\bfe} (t) = {\cat}_{e_1}(f_1) \boxtimes \cdots \boxtimes {\cat}_{e_k}(f_k)$. To see this, we prove that both sides coincide on product elements $D = D_1 \ootimes D_k \in \Sym^{e_1} V_1^*\ootimes \Sym^{e_k} V_k^*$. Regard $ D \in \Sym^{e_1}V_1^* \otimes \ldots \otimes \Sym^{e_k}V_k^*$ as the differential operator $D = D_1 \cdots D_k \in \Sym^{\bullet}(V_1^* \ooplus V_k^*)$ and $t = f_1 \cdots f_k \in \Sym^{\bullet}(V_1 \ooplus V_k)$, where the factors of $t$ involve disjoint sets of variables and the factors of $D$ act on disjoint sets of variables. Then, we obtain
\begin{align*}
\cat_{\bfe} (t) ( D_1 \ootimes D_k) &= (D_1 \ootimes D_k) (t) \\ 
& = (D_1 \cdots D_k) ( f_1 \cdots f_k) \\ 
& =D_1(f_1)  \cdots D_k(f_k) \\ 
&= D_1(f_1)  \ootimes  D_k(f_k) \\ 
&=  {\cat}_{e_1}(f_1)(D_1) \ootimes  {\cat}_{e_k}(f_k) (D_k) \\  
&= \left[ {\cat}_{e_1}(f_1) \bboxtimes {\cat}_{e_k}(f_k) \right] (D_1 \ootimes D_k),
\end{align*}
  and therefore $ \cat_{\bfe} (t) = {\cat}_{e_1}(f_1) \bboxtimes {\cat}_{e_k}(f_k).$
\end{proof}
\begin{example}
	Let $m_1 \otimes m_2 = x_{1,0}^{a_0}\cdots x_{1,n_1}^{a_{n_1}} \otimes x_{2,0}^{b_0}\cdots x_{2,n_2}^{b_{n_2}} \in \Sym^{a}V_1 \otimes \Sym^{b}V_2$, where $a = \sum_i a_i$ and $b = \sum_i b_i$. Then, $$(m_1\otimes m_2)^\perp = \left(\partial_{1,0}^{a_0+1},\ldots,\partial_{1,n_1}^{a_{n_1}+1}\right) + \left(\partial_{2,0}^{b_0+1},\ldots,\partial_{2,n_2}^{b_{n_2}+1}\right) \subset \Sym^\bullet (V_1^* \oplus V_2^*).$$
\end{example}

\subsection{Multiplicative lower bounds and tensor asymptotic rank}\label{subsec: tensor asymptotic rank}

In \cite{ChrGesJen:BorderRankNonMult}, the notion of \emph{tensor asymptotic rank} was introduced. We recall the definition in the case of Waring rank. 

Let $f \in \Sym^d V$. The \textit{tensor asymptotic rank} of $f$ is 
\begin{equation}\label{eqn: def asymptotic rank}
\aR_{d}(f) = \lim_{k \to \infty} \left[ \rmR_{\bfd} (f^{\otimes k}) \right]^{1/k}.
\end{equation}

The limit in \eqref{eqn: def asymptotic rank} exists by Fekete's Lemma (see, e.g., \cite[pg. 189]{PolSze:ProblemsTheoremsAnalysisI}) via the submultiplicativity properties of Waring rank under tensor product and in fact it is an infimum over $k$. As a consequence, multiplicative lower bounds for $\rmR_{d}(f)$ are lower bounds for $\aR_{d}(f)$; see Lemma \ref{lemma:MultiplicativeLowerBound} below. This was observed in \cite{ChrJenZui:NonMultTensorRank} for (generalized) flattening lower bounds, but the general result holds for any multiplicative lower bound.

\begin{lemma}\label{lemma:MultiplicativeLowerBound}
Let $f \in \Sym^dV$. Assume that $M$ is a multiplicative lower bound for the rank of $f$, i.e., $\rmR_{\bfd}(f^{\otimes k}) \geq M^k$, for every $k \geq 1$. Then, $\aR_d(f) \geq M$.
\end{lemma}
\begin{proof}
Consider the inequality $\rank_{\bfd}(f^{\otimes k}) \geq M^k$. Raise both sides to the $1/k$ to obtain the inequality $[\rank_{\bfd}(f^{\otimes k})]^{1/k} \geq M$. Conclude by passing to the limit as $k \to \infty$.
\end{proof}

The definition of the tensor asymptotic rank was inspired by the similar definition of asymptotic rank given by Strassen in the setting of tensors \cite{Str86:Asymptotic} in terms of tensor Kronecker (or flattened) product. In \cite{Bini:RelationsExactApproxBilAlg}, Bini proved that the growth of the tensor rank under Kronecker powers is essentially the same as the growth of the border rank under Kronecker powers, so that the definition of asymptotic rank of a tensor, in the sense of Strassen, can equivalently be given in terms of rank or border rank. The analogous result holds for the tensor asymptotic rank of \eqref{eqn: def asymptotic rank}, as proved in \cite[Theorem 8]{ChrJenZui:NonMultTensorRank} in the case of tensors and in \cite[Proposition 6.2]{ChrGesJen:BorderRankNonMult} in full generality. In particular the border rank of a homogeneous form is an upper bound for its tensor asymptotic rank.

\begin{proposition}[{\cite[Proposition 6.2]{ChrGesJen:BorderRankNonMult}}]\label{prop:LowerBound}
	Let $f \in \Sym^dV$. Then, 
	\[
\aR_{d}(f)~\leq~\uR_{d}(f).	 
	\]
\end{proposition}

One can define a ``cactus analog'' of these notions. Let $f \in \Sym^d V$. The \emph{tensor asymptotic cactus rank} of $f$ is 
\begin{equation}\label{eqn: def asymptotic cactus rank}
\acR_{d}(f) = \lim_{k \to \infty} \left[ \cR_{\bfd} (f^{\otimes k}) \right]^{1/k}.
\end{equation}
Fekete's Lemma and the submultiplicative properties of cactus rank guarantee that this limit exists as well. Moreover, the analog of Lemma \ref{lemma:MultiplicativeLowerBound} holds.

The proof that the analog of \cite[Proposition 6.2]{ChrGesJen:BorderRankNonMult} holds for cactus varieties is slightly more delicate. The difficulty is caused by the fact that cactus varieties are in general reducible and not equidimensional: the argument is essentially the same, with the only modification that one considers a single irreducible component of the cactus variety.

\begin{remark}
In \cite[Proposition 6.2]{ChrGesJen:BorderRankNonMult}, it was claimed that the set $\sigma^\circ_r(X) = \{ q : \uR_X(q) = \rmR_X(q) = r\}$ is a Zariski-open subset in $\sigma_r(X)$. This is not true in general, even in the case of Veronese varieties. However, by Chevalley's Theorem, $\sigma^\circ_r(X)$ is a \emph{constructible} set, in the sense of \cite[\S2.C]{Mum:ComplProjVars}, and in particular it contains a Zariski-open subset of $\sigma_r(X)$. The proof is not affected by this oversight.
\end{remark}

We give some additional details for the generalization of \cite[Proposition 6.2]{ChrGesJen:BorderRankNonMult}, pinpointing the difference between the setting of secant varieties and the one of cactus varieties.
\begin{proposition}\label{prop: lower bound to border cactus}
Let $f \in \Sym^dV$. Then, 
\[
\acR_{d}(f) \leq \ucR_{d}(f).	 
\]
\end{proposition}
\begin{proof}
Following the same argument as \cite[Proposition 6.2]{ChrGesJen:BorderRankNonMult}, we determine a constant $e$ such that, for every $k$, $\cR_{\bfd}(f^{\otimes k}) \leq \ucR_d(f)^k (ek + 1)$.

Let $r = \ucR_d(f)$ so that $[f] \in \frakk_r(\calV_{d,n})$. Let $Z$ be an irreducible component of $\frakk_r(\calV_{d,n})$ such that $[f] \in Z$ and let $Z^\circ = \{ g \in Z: \cR_d(g) = \ucR_d(g) = r\}$. The set $Z^\circ$ is constructible and in particular it contains a Zariski-open subset of $Z$. 

At this point the proof follows exactly the same construction as in \cite{ChrGesJen:BorderRankNonMult}. We quickly sketch the rest of the argument. One considers a generic linear space $L$ through $[f]$ with $\dim L = \codim Z +1$. The variety $L \cap Z$ is a (possibly reducible) curve in $Z$ and by the genericity of $L$ every component contains a Zariski-open subset of points in $Z^\circ$. Let $E$ be an irreducible curve in $L \cap Z$ such that $[f] \in E$ and let $e = \deg(E)$. Hence, $[f^{\otimes k}] \in \nu_k( E) \subseteq \bbP \Sym^k ( \Sym ^d V) \subseteq  \bbP ( \Sym ^d V)^{\otimes k}$ and $\deg (\nu_k( E)) = ek$. A set of $ek+1$ generic points in $\nu_k(E)$ span the entire $\langle \nu_k( E ) \rangle$ and in particular they span $[f^{\otimes k}]$. For generic points $[g^{\otimes k}] \in \nu_k(E)$, one has $\cR_{\bfd}(g^{\otimes k}) \leq \cR_{d}(g)^k = r^k$ by submultiplicativity. We conclude $\rmR_\bfd(f) \leq (ek+1)r^k$, as desired.
\end{proof}

\section{Ranestad-Schreyer lower bound}\label{sec: RS}
In this section, we focus on the Ranestad-Schreyer lower bound for cactus rank, proved in \cite[Proposition~1]{RanSch:RankSymmetricForm} and on the generalization proposed by Teitler in \cite{Tei14:GeometricBounds}.

We will show that these lower bounds extend to border cactus rank and thus to border rank under the hypothesis that the tensor of interest, and its tensor powers satisfy $(\star_c)$.

For any homogeneous polynomial $f \in \Sym^dV$, denote by $A_f$ the quotient algebra $\Sym^\bullet V^* / f^\perp$. Since $(f^\perp)_e = \Sym^e V$ for $e > d$, we deduce that $A_f$ is a finite dimensional vector space, and, since $f^\perp$ is a homogeneous ideal, the quotient algebra inherits the grading so that $A_f = \bigoplus _{e=0}^d (A_f)_e$. 

Similarly, if $t \in \Sym^{d_1}V_1 \ootimes \Sym^{d_k}V_k$, define $A_t = \Sym^\bullet (V_1^* \ooplus V_k^*) / t^\perp$, which is again multigraded and  finite dimensional because $t^\perp$ is a multihomogeneous ideal and $(t^\perp)_{\bfe} = 0$, if $e_j > d_j$ for at least one $j$. In particular, $A_t = \bigoplus_{\substack{e_i = 0 \vvirg d_i \\ i = 1,\ldots,k}} (A_t)_{\bfe}$.

The Ranestad-Schreyer lower bound for homogeneous forms is as follows:
\begin{proposition}[{\cite[Proposition~1]{RanSch:RankSymmetricForm}}]
Let $f \in S^d V$ be a homogeneous polynomial such that $f^\perp$ is generated in degree $\delta$. Then
\[
 \rmR_{\bfd}(f) \geq \cR_{\bfd}(f) \geq \frac{\dim A_f}{\delta}.
\]
\end{proposition}

In \cite{Tei14:GeometricBounds}, Teitler proposed a generalization to partially symmetric tensors. We do not know whether this generalization holds in general; conditions $(\star)$ and $(\star_c)$ introduced in Section \ref{sec: intro} identify exactly those tensors for which the generalization is true. 

The following result related the apolar ideal of a tensor product with the apolar ideals of its factors and correspondingly their apolar algebras.

\begin{proposition}\label{prop: RS multiplicative}
 Let $t = f_1 \ootimes f_k$, with $f_i \in \Sym^{d_i}V_i$. Then $A_t \simeq A_{f_1} \ootimes A_{f_k}$ as multigraded algebras. In particular, 
 \[
\dim A_t = (\dim A_{f_1}) \cdots (\dim A_{f_k}).  
 \]
 Moreover, if $f_i^\perp$ is generated in degree at most $\delta_i$, for $i = 1,\ldots,k$, then $t^\perp$ is generated in degree at most $\boldsymbol{\delta} = (\delta_1 \vvirg \delta_k)$.
 \end{proposition}
 \begin{proof}
For any $i = 1 \vvirg k$, since $(f_i^\perp)^{ext} \subseteq t^\perp$, the inclusion $\Sym^\bullet V_i^* \to \Sym^\bullet( V_1 ^* \ooplus V_k^*)$ descends to the quotient algebras, providing a graded algebra homomorphism $\phi_i: A_{f_i}  \to A_t$. Notice that elements of degree $e$ are mapped to elements of multidegree $(0 \vvirg 0,e,0\vvirg 0)$, where $e$ is at the $i$-th entry. 

By the universal property of tensor products, the $\phi_i$'s lift to a homomorphism of graded algebras
\[
 \phi: A_{f_1} \ootimes A_{f_k} \to A_t
\]
defined by $\phi( g_1 \ootimes g_k) = \phi_1(g_1) \cdots \phi_k(g_k) = g_1 \cdots g_k$ on product elements and extended linearly. Notice that $\phi$ is surjective, because the algebra $A_t$ is generated by elements of total degree equal to one, i.e., the images of the variables $x_{i,j}$, and those are in the image of $\phi$.

We conclude that $\phi$ is an isomorphism by showing that the multigraded components of $ A_{f_1} \ootimes A_{f_k}$ and $A_t$ have the same dimension. Indeed, in multidegree $\bfe$, we have
\begin{align*}
\dim (A_{f_1} \ootimes A_{f_k})_{\bfe} &=  (\dim (A_{f_1})_{e_1} \cdots (\dim (A_{f_k})_{e_k}  \\ &= \rk(\cat_{e_1}(f_1)) \cdots  \rk(\cat_{e_k}(f_k))
\end{align*}
and 
\[
\dim (A_{t})_{\bfe} =  \rk(\cat_{\bfe}(f_1 \ootimes f_k)).
\]
Recall that $\cat_{e_1}(f_1)) \boxtimes \cdots \boxtimes \cat_{e_k}(f_k) = \cat_{\bfe}(f_1 \ootimes f_k)$ (Lemma \ref{lemma:apolar}) and then apply Lemma \ref{lemma:Kronecker}(i).

The second part of the statement is also immediate from Lemma \ref{lemma:apolar}: we can see from the equality $t^\perp =  (f_1^\perp)^{ext} + \ldots + (f_k^\perp)^{ext}$ that $t^\perp$ is generated in multidegrees of the form $(0 \vvirg 0 , \delta_i , 0 \vvirg 0)$, for all $i = 1,\ldots,k$.
 \end{proof}

If $t = f_1 \ootimes f_k$ satisfies $(\star)$ or $(\star_c)$, then Proposition \ref{prop: RS multiplicative} immediately provides the following lower bound for its rank or cactus rank. 
\begin{corollary}\label{corol: multiplicative RS}
Let $f_i \in \Sym^{d_i}V$, for $i = 1,\ldots,k$, and assume that $f_i^\perp$ is generated in degree at most $\delta_i$. If $f_1 \ootimes f_k$ satisfies $(\star)$, then,
\[
	\rmR_{\bfd}(f_1\ootimes f_k) \geq \frac{\dim(A_{f_1})\cdots\dim(A_{f_k})}{\delta_1\cdots\delta_k}
\]
and if $f_1 \ootimes f_k$ satisfies $(\star_c)$, then the lower bound holds also for cactus rank.

In particular, if $f \in \Sym^dV$, $f^\perp$ is generated in degree $\delta$, and $f^{\otimes k}$ satisfies $(\star)$, then, 
\[
 \rmR_{\bfd}(f^{\otimes k}) \geq \left(\frac{\dim A_f}{\delta}\right)^k, \qquad \text{for every }k \geq 1.
\]
and if $f^{\otimes k}$ satisfies $(\star_c)$ the lower bound holds also for cactus rank.
\end{corollary}

A consequence of these results is that every tensor for which the rank is multiplicative under tensor product satisfies $(\star)$ together with all its tensor powers. Similarly every tensor for which the cactus rank is multiplicative under tensor product satisfies $(\star_c)$ together with all its tensor powers. In particular all monomials for which the flattening lower bounds from \cite{Oe16:BorderMonomials} hold and all the tensor products of any number of them satisfy $(\star)$.

The following completes the proof of Theorem \ref{thm: RS for border rank}:
\begin{proof}[Proof of Theorem \ref{thm: RS for border rank}]
If $f^{\otimes k}$ satisfies $(\star)$ for every $k$, by Corollary \ref{corol: multiplicative RS} we obtain the lower bound $\rmR_{\bfd}(f^{\otimes k}) \geq \left(\frac{\dim A_f}{\delta}\right)^k$, for every $k$. In particular, this is a multiplicative lower bound.

Applying Lemma \ref{lemma:MultiplicativeLowerBound} to this multiplicative lower bound, we obtain that $\aR_d(f) \geq \acR_d(f) \geq \frac{\dim  A_f }{\delta}$. Therefore, by Proposition \ref{prop:LowerBound} and Proposition \ref{prop: lower bound to border cactus}, we deduce that $\uR_d(f) \geq \ucR_d(f) \geq \frac{\dim  A_f }{\delta}$.

The same argument, applied in the case where $f^{\otimes k}$ satisfies $(\star_c)$ for every $k$, provides the lower bound for border cactus rank.
\end{proof}

{\small
\subsection*{Acknowledgements} 
M.C. and F.G. acknowledge financial support from the VILLUM FONDEN via the QMATH Centre of Excellence (Grant no. 10059). A.O. acknowledges financial support from the Alexander von Humboldt-Stiftung (Germany) via a Humboldt Research Fellowship for Postdoctoral Researchers (April 2019 - March 2021). We thank J. Buczy\'nski, M. Ga\l\k{a}zka, L. Oeding, G. Ottaviani and Z. Teitler for their helpful comments on the earlier versions of the paper.
}

\bibliographystyle{alpha}
\bibliography{bibMonom}
  \end{document}